\documentclass[reqno, 12pt]{amsart}
\pdfoutput=1
\makeatletter
\let\origsection=\section \def\section{\@ifstar{\origsection*}{\mysection}} 
\def\mysection{\@startsection{section}{1}\z@{.7\linespacing\@plus\linespacing}{.5\linespacing}{\normalfont\scshape\centering\S}}
\makeatother        

\usepackage{amsmath,amssymb,amsthm}
\usepackage{mathrsfs}
\usepackage{mathabx}\changenotsign
\usepackage{dsfont}

\usepackage{todonotes}
\usepackage{verbatim}
 
\usepackage{graphicx}
\usepackage{xcolor}
\usepackage{tikz}
\usetikzlibrary{arrows}

\usepackage{hyperref}
\hypersetup{
    colorlinks,
    linkcolor={red!60!black},
    citecolor={green!60!black},
    urlcolor={blue!60!black},
    pdftitle={Infinite thin linear equation systems},
    pdfauthor={J. Pascal Gollin and Attila Jo\'{o}}
}

\usepackage[open,openlevel=2,atend]{bookmark}

\usepackage[maxbibnames=99]{biblatex}
\usepackage[T1]{fontenc}
\usepackage{lmodern}
\usepackage[babel]{microtype}
\usepackage[english]{babel}

\usepackage{geometry}
\numberwithin{equation}{section}
\numberwithin{figure}{section}

\usepackage{enumitem}

\let\polishlcross=\l
\def\l{\ifmmode\ell\else\polishlcross\fi}

\def\paragraph#1{%
  \noindent\textbf{#1.}\enspace}

\let\emptyset=\varnothing
\let\setminus=\smallsetminus

\makeatletter
\def\moverlay{\mathpalette\mov@rlay}
\def\mov@rlay#1#2{\leavevmode\vtop{   \baselineskip\z@skip \lineskiplimit-\maxdimen
   \ialign{\hfil$\m@th#1##$\hfil\cr#2\crcr}}}
\newcommand{\charfusion}[3][\mathord]{
    #1{\ifx#1\mathop\vphantom{#2}\fi
        \mathpalette\mov@rlay{#2\cr#3}
      }
    \ifx#1\mathop\expandafter\displaylimits\fi}
\makeatother

\DeclareFontFamily{U}  {MnSymbolC}{}
\DeclareSymbolFont{MnSyC}         {U}  {MnSymbolC}{m}{n}
\DeclareFontShape{U}{MnSymbolC}{m}{n}{
    <-6>  MnSymbolC5
   <6-7>  MnSymbolC6
   <7-8>  MnSymbolC7
   <8-9>  MnSymbolC8
   <9-10> MnSymbolC9
  <10-12> MnSymbolC10
  <12->   MnSymbolC12}{}
\DeclareMathSymbol{\powerset}{\mathord}{MnSyC}{180}

\let\epsilon=\varepsilon

\let\rho=\varrho
\let\theta=\vartheta
\let\phi=\varphi

\theoremstyle{plain}
\newtheorem{thm}{Theorem}[section]
\newtheorem{theorem}[thm]{Theorem}
\newtheorem{lemma}[thm]{Lemma}

\newtheorem*{claim*}{Claim}

\newtheorem{thm-intro}{Theorem}[]
\newtheorem{cor-intro}[thm-intro]{Corollary}
\newtheorem{conj-intro}[thm-intro]{Conjecture}
\newtheorem{question-intro}[thm-intro]{Question}
\newtheorem*{meta-question*}{Meta question}

\theoremstyle{definition}

\newtheorem{obs}[thm]{Observation}

\newtheorem*{example*}{Example}

\newcommand{\abs}[1]{\ensuremath{{\lvert {#1} \rvert}}}

\begin{document}

\author[J.~P.~Gollin]{J.~Pascal Gollin}
\address{J. Pascal Gollin, Discrete Mathematics Group, Institute for Basic Science (IBS), 55 Expo-ro, Yuseong-gu, Daejeon, 
Korea, 34126}
\email{\tt pascalgollin@ibs.re.kr}
\thanks{The first author was supported by the Institute for Basic Science (IBS-R029-Y3).}

\author[A.~Jo\'{o}]{Attila Jo\'{o}}
\address{Attila Jo\'{o},  Department of Mathematics, University of Hamburg, Bundesstra{\ss}e 55 (Geomatikum), 20146 
Hamburg, Germany and Set theory and general topology research division,
    Alfr\'{e}d R\'{e}nyi Institute of Mathematics,  13-15 Re\'{a}ltanoda St., 
    Budapest, Hungary}
\email{\tt attila.joo@uni-hamburg.de}
\thanks{Funded by the Deutsche Forschungsgemeinschaft (DFG, German
Research Foundation)-513023562 and partially by NKFIH OTKA-129211}

\title[Matching variables to equations in infinite linear equation systems]{Matching variables to equations in infinite linear equation systems}

\date{November 23}

\begin{abstract}
    A fundamental result in linear algebra states that if a homogenous linear equation system has only the trivial solution, then there are at most as many variables as equations. 
    We prove the following generalisation of this phenomenon. 
    If a possibly infinite homogenous linear equation system with finitely many variables in each equation has only the trivial solution, then there exists an injection from the variables to the equations that maps each variable to an equation in which it appears. 
\end{abstract}

\maketitle

\section{Introduction}

Infinite linear equation systems appear in the most diverse areas of mathematics. 
They have a key role in boundary value problems for linear partial differential equations. 
Banach devoted them in his book~\cite{banach1987book} a whole section entitled ``Systems of linear equations in infinitely 
many unknowns''. 
In this setting, there are countably many variables and infinite sums are defined via convergence. 
Another possible approach deals with sums that are ``thin'', i.e.~that there are only finitely many non-zero summands in each. 
Such sums appear for example in horizon planning programs (see \cite{romeijn1998shadowprices}). 
More recently, these sums turned out to be fruitful in the representation theory of infinite matroids. 
For a set~$I$ and field~$\mathbb{F}$, a family~${\mathcal{F} = ( f_j \in \mathbb{F}^I \,\colon\, j \in J )}$ is called \emph{thin} if for each~${i \in I}$ there are only finitely many~${j \in J}$ with~${f_j(i) \neq 0}$. 
Infinite linear combinations of the functions~$f_j$ can be defined in a natural way. 
Indeed, if~${\lambda_j \in \mathbb{F}}$ for~${j \in J}$, then for each~${i \in I}$ the sum~${\sum_{j \in J} \lambda_j f_j(x)}$ is a well-defined element of~${\mathbb{F}}$, therefore~${\sum_{j \in J} \lambda_j f_j}$ can be considered as an element of~$\mathbb{F}^{I}$. 
If the constant~$0$ function on~$I$ is obtained only if~${\lambda_j = 0}$ for every~${j \in J}$, then~$\mathcal{F}$ is said to be \emph{thinly independent}. 
In other words, $\mathcal{F}$ is thinly independent if the (possibly infinite) homogenous linear equation system~${\sum_{j \in J} f_j(i) x_j = 0\ (i \in I)}$ has only the trivial solution. 
Investigation of the concept of thin dependence in the context of matroid theory was initiated by Bruhn and Diestel~\cite{bruhn2011infmatrgraph} and became a relatively well-understood subject after the discoveries of Afzali and Bowler~\cite{borujeni2016finitary}.  

Our main result states that for a thinly independent~$\mathcal{F}$, there is always a system of distinct representatives for the family~${\{ \mathsf{supp}(f_j) \, \colon \, j \in J \}}$ where~${\mathsf{supp}(f_j)}$ denotes the \emph{support} of~$f_j$, that is the set of those elements of~$X$ on which~$f_j$ is non-zero. 
Note that by considering arbitrary families with the the usual linear independence in~$\mathbb{F}^I$, the analogous statement fails. 
Indeed, for example the dimension of the vector space~${\mathsf{GF}(2)^{\mathbb{N}}}$ is continuum and therefore no base of it can be injectively mapped into~$\mathbb{N}$. 
Considering thin families but still the usual independence does not fix this issue. 
To demonstrate this, let us take in~$\mathsf{GF}(2)^{\mathbb{N}}$ the unit vectors together with their thin sum, the constant~$1$ vector. 
Then, no desired injection exists, although this family is linearly independent. 

Let us rephrase our main result in a more elementary way. 
A basic fact in linear algebra states that if a finite homogeneous linear equation system has only the trivial solution, then there are at most as many variables as equations. 
Naively lifting observations such as this to an infinite setting often loses some interesting structural information about the problem, since a pure comparison between cardinalities of sets is a rough measure. 
Instead, let us strengthen the fact to include more structural information. 
If a finite homogeneous linear equation system has only the trivial solution, then it is not too hard to show (using standard techniques from matching theory) that there exists an injection from the variables to the equations where each variable is mapped to an equation in which it has non-zero coefficient. 
Our main result states that this remains true for every homogeneous thin linear equation system. 

\begin{thm-intro}\label{thm: main}
    Let~$I$ and~$J$ be sets, let~$\mathbb{F}$ be a field, and let~${a_{i,j} \in \mathbb{F}}$ for~${i \in I}$ and~${j \in J}$ such that for each~${i \in I}$ there are only finitely many~${j \in J}$ with~${a_{i,j} \neq 0}$. 
    Suppose that the (possibly infinite) homogeneous linear equation system
    \begin{equation}
        \label{eq: thinLES}\tag{$\ast$}
        \sum_{j \in J} a_{i,j} x_j = 0 \quad (i \in I)
    \end{equation}
    has only the trivial solution. 
    Then there is an injection~${\varphi \colon J \to I}$ such that~${a_{\varphi(j),j} \neq 0}$ for every~${j \in J}$. 
\end{thm-intro}

In a regular matrix, one can rearrange its rows to obtain a matrix in which every entry of the diagonal is non-zero. Using the terminology of thinly independent families as before and using the Cantor-Bernstein Theorem (see Theorem~\ref{thm: cantor-bernstein}), we obtain the following generalisation of this fact as a corollary of Theorem~\ref{thm: main}. 

\begin{cor-intro}
    \label{cor: diagonal}
    Let~$I$ and~$J$ be sets, let~$\mathbb{F}$ be a field, and let~${a_{i,j} \in \mathbb{F}}$ for~${i \in I}$ and~${j \in J}$ such that the families ${( ( i \mapsto a_{i,j} ) \in \mathbb{F}^I \, \colon \, j \in J)}$ and ${( ( j \mapsto a_{i,j} ) \in \mathbb{F}^J \, \colon \, i \in I )}$ are both thinly independent. 
    Then there is a bijection~${\psi \colon J \to I}$ such that~${a_{\psi(j),j} \neq 0}$ for every~${j \in J}$. 
\end{cor-intro}

\section{Notation and Preliminaries}
\label{sec: prelims}

For the domain and range of a function~$f$ we write~${\mathsf{dom}(f)}$ and~${\mathsf{ran}(f)}$ respectively. We write ${\mathsf{ran}_I(f)}$ as an abbreviation of ${\mathsf{ran}(f)}\cap I$.
For a subset~${S \subseteq \mathsf{dom}(f)}$, we denote by~$f {\upharpoonright} S$ the \emph{restriction} of~$f$ to~$S$. 

A \emph{bipartite graph}~$G$ is a triple~$(S,T,E)$, where~$S$ and~$T$ are disjoint sets \linebreak and~${E \subseteq \{\{ s,t \} \, \colon \,  s \in S \textnormal{ and } t \in T\}}$. 
The elements of~${S \cup T}$ are the \emph{vertices} of~$G$ and the elements of~$E$ are the \emph{edges} of~$G$. 
The set containing all~${w \in S \cup T}$ for which ~${\{v,w\} \in E}$ is the \emph{neighbourhood}~${N_G(v)}$ of~$v$, and the cardinal~$\abs{N_G(v)}$ is the \emph{degree} of~$v$. 
A \emph{matching}~$M$ in~$G$ is a set of edges no two of which share a vertex. 
We say a matching~$M$ \emph{covers} a set~${X \subseteq S \cup T}$ if each vertex in~$X$ is contained in some edge in~$M$. 
A matching is \emph{perfect} if it covers~${S \cup T}$. 

Let~$I$ and~$J$ be sets and let~$\mathbb{F}$ be a field. 
We denote by~$\mathbb{F}^I$ the vector space of functions from~$I$ to~$\mathbb{F}$. 
Given an element~${b \in \mathbb{F}^I}$ and~${i \in I}$ we write~$b_i$ instead of~${b(i)}$. 
A \emph{matrix} in this paper is a function~${A \colon I \times J \to \mathbb{F}}$. 
For~${i \in I}$ and~${j \in J}$, we write~$a_{i,j}$ instead of~${A(i,j)}$. 
For a fixed~${i \in I}$, the map~${j \mapsto a_{i,j}}$ is the \emph{row} of~$A$ corresponding to~$i$ while columns are defined analogously. 
The \emph{rank}~${r(A)}$ of a finite matrix~$A$ is the dimension of the subspace of~${\mathbb{F}^{J}}$ spanned by its rows (equivalently the dimension of the subspace of~${\mathbb{F}^{I}}$ spanned by its columns). 

Let~${A \colon I \times J \to \mathbb{F}}$ be a matrix, and let~${b \in \mathbb{F}^I}$. 
We say that~$A$ is \emph{row-thin} if the support of each row of~$A$ is finite. 
If~$A$ is row-thin, then we denote by
\[
    \sum_{j \in J} a_{i,j} x_j = b_i \quad (i \in I)
\]
a \emph{thin system of linear equations} with \emph{variables}~${x_j \ (j \in J)}$. 
We may also denote this system by~${Ax = b}$. 
If~${b_i = 0}$ for all~${i \in I}$, we call the system \emph{homogeneous}. 
Note that given an element~${s \in \mathbb{F}^J}$, the sum~$\sum_{j \in J} a_{i,j} s_j$ is a well-defined element of~$\mathbb{F}$. 
A \emph{solution} for~${Ax = b}$ is an element~${\lambda \in \mathbb{F}^J}$ such that~$\sum_{j \in J} a_{i,j} \lambda_j = b_i$ for each~${i \in I}$. 
We say that~${Ax = b}$ is \emph{solvable} if it has a solution. 

If the field~$\mathbb{F}$ is finite, then standard compactness arguments show that the solvability of a thin linear equation system is equivalent with the solvability of all its finite subsystems. 
Maybe surprisingly, this remains true without any restriction on~$\mathbb{F}$. 

\begin{theorem}[Compactness of thin linear equation systems, Cowen and Emerson \cite{cowen1996complin}\footnotemark]
    \label{thm: compact LES}
    If every finite subset of the equations of a thin linear equation system is solvable, then the whole system is solvable.
\end{theorem}

\footnotetext{This theorem was rediscovered independently by Bruhn and Georgakopoulos~\cite{bruhn2011bases}. 
Their proof was later simplified by Afzali and Bowler {\cite[Lemma 4.2]{borujeni2015thin}}.}

To prove our main theorem, we also need a tool from infinite matching theory developed by Wojciechowski~\cite{wojciechowski1997criteria}. 
Let~${G = (S,T,E)}$ be a bipartite graph. 
A \emph{string} corresponding to~$G$ is an injective function~$f$ defined on an ordinal number with range~${\mathsf{ran}(f) \subseteq S \cup T}$. 
A string~$f$ is called \emph{saturated} if whenever~$f(\alpha) = v \in S$, then~${N_G(v) \subseteq \mathsf{ran}(f {\upharpoonright} \alpha})$. 
In other words, a vertex~${v \in S}$ can only appear in the transfinite sequence~$f$ after all of its neighbours already appeared. 
For a saturated string~$f$ with~${\mathsf{dom}(f) = \alpha}$, the quantity~${\mu_G(f) \in \mathbb{Z} \cup \{ -\infty, +\infty \}}$ is defined by transfinite recursion on~$\alpha$ as follows. 

\[
    \mu_G(f) := 
    \begin{cases} 0 &\mbox{if } \alpha=0, \\
        \liminf \{ \mu_G(f {\upharpoonright} \beta) \, \colon \, \beta<\alpha \}  & \mbox{if } \alpha \text{ is a limit ordinal},\\
        \mu_G(f {\upharpoonright} \beta)-1 & \mbox{if } \alpha=\beta+1 \text{ and } f(\beta)\in S,\\
        \mu_G(f {\upharpoonright} \beta)+1 & \mbox{if } \alpha=\beta+1 \text{ and } f(\beta)\in T, 
    \end{cases}
\]
where we use the convention that~${\pm \infty + k = \pm \infty}$ for~${k \in \mathbb{Z}}$. 
It is not too hard to prove that if~$G$ admits a matching that covers~$S$, then we must have~${\mu_G(f) \geq 0}$ for every saturated string~$f$. 
Under some assumption the reverse is also true. 

\begin{thm}[Wojciechowski {\cite[Theorem 1]{wojciechowski1997criteria}}\footnotemark]
    \label{thm: jerzy}
    Let~${G = (S,T,E)}$ be a bipartite graph in which each vertex in~$T$ has countable degree. 
    Then there is a matching in~$G$ covering~$S$ if and only if~${\mu_G(f) \geq 0}$ for every saturated string~$f$ corresponding to~$G$
\end{thm}

\footnotetext{The criterion given in Theorem~\ref{thm: jerzy} is called $\mu$\nobreakdash-admissibility and was inspired by the $q$\nobreakdash-admissibility criterion of Nash-Williams (see \cite{nash1978another}). 
A characterisation of matchability for arbitrary bipartite graphs was discovered by Aharoni, Nash-Williams and Shelah~\cite{aharoni1983general}. 
For a survey on infinite matching theory (including the non-bipartite case) we refer to~\cite{aharoni1991infinite}.}

To obtain Corollary~\ref{cor: diagonal}, we now state the well-known theorem of Cantor and Bernstein in a stronger, graph-theoretic form. 

\begin{thm}[Cantor-Bernstein \cite{cantor1987}]
    \label{thm: cantor-bernstein}
    If ${G = (S,T,E)}$ is a bipartite graph and there exist a matching~$M_S$ that covers~$S$ as well as a matching~$M_T$ that covers~$T$, then~$G$ admits a perfect matching. 
\end{thm}

\section{Proof of the main results}

Let us fix a homogeneous thin linear equation system~${Ax = 0}$, where~${A \colon I \times J \to \mathbb{F}}$, 
that admits only the trivial solution. 
Without loss of generality, we may assume that~$I$ and~$J$ are disjoint. 
We define a bipartite graph~$G_A = (J,I,E)$ where~${\{i,j\} \in E}$ for~${i \in I}$ and~${j \in J}$ if and only if~$a_{i,j}$ is non-zero. 
We will simply write~$\mu$ instead of~$\mu_{G_A}$. 
Moreover, when we refer to a saturated string we will always mean a saturated string with respect to~$G_A$. 

\begin{obs}
    \label{obs: saturated}
    If a string~$f$ is saturated, then for~${i \in I \setminus (\mathsf{ran}_I(f)})$ and~${j \in \mathsf{ran}_J(f)}$, we have~${a_{i,j} = 0}$. 
    In other words, the matrix~${A {\upharpoonright} (I \times (\mathsf{ran}_J(f)))}$ is obtained from the matrix~${A {\upharpoonright} ((\mathsf{ran}_I(f)) \times (\mathsf{ran}_J(f)))}$ by extending the columns by zeroes. 
\end{obs}

\begin{lemma}
    \label{lem: main}
    If~$f$ is a saturated string such that~$\mu$ takes non-negative finite values on all proper initial segments of~$f$, 
    then for every finite~${I_0 \subseteq \mathsf{ran}_I(f)}$, there is a finite~${I' \subseteq \mathsf{ran}_I(f)}$ extending~${I_0}$ and a finite~${J' \subseteq \mathsf{ran}_J(f)}$ such that~${\mu(f) = \abs{I'} - r(A {\upharpoonright} (I' \times J'))}$.  
\end{lemma}

\begin{proof}
    We apply transfinite induction on~${\mathsf{dom}(f) =: \alpha}$. 
    
    If~${\alpha = 0}$, we must have~${I_0 = \emptyset}$ and we can only take~${I' := J' := \emptyset}$. 
    This is appropriate because~${\mu(\emptyset) = \abs{\emptyset} = r(\emptyset) = 0}$. 
    
    If~$\alpha$ is a limit ordinal, then~${U := \{ \beta < \alpha \, \colon \, \mu(f {\upharpoonright} \beta) = \mu(f) \}}$ is unbounded in~$\alpha$ by the definition of~$\mu$. 
    Let a finite set~${I_0 \subseteq \mathsf{ran}_I(f)}$ be given and let~${\beta \in U}$ large enough to satisfy~${I_0 \subseteq \mathsf{ran}(f {\upharpoonright} \beta)}$. 
    By induction, 
    we obtain a finite set~${I' \subseteq \mathsf{ran}_{I}(f {\upharpoonright} \beta)}$ extending~$I_0$ and a finite set~${J' \subseteq \mathsf{ran}_J(f {\upharpoonright} \beta)}$ such that~${\mu(f {\upharpoonright} \beta) = \abs{I'} - r(A {\upharpoonright} (I' \times J'))}$. 
    Since ${\mathsf{ran}(f {\upharpoonright} \beta) \subseteq \mathsf{ran}(f)}$, 
    these~$I'$ and~$J'$ are as desired. 
    
    Finally, assume that~${\alpha = \beta+1}$ and let a finite~${I_0 \subseteq \mathsf{ran}_I(f)}$ be given. 
    Suppose first that~${i := f(\beta) \in I}$. 
    We may assume without loss of generality that~${i \in I_0}$. 
    By applying the induction hypotheses for~${f {\upharpoonright} \beta}$ and~${I_0 \setminus \{ i \}}$, 
    we can pick finite 
    sets~${I^{*} \subseteq \mathsf{ran}_I(f {\upharpoonright} \beta)}$ 
    and~${J^{*} \subseteq  \mathsf{ran}_J(f {\upharpoonright} \beta)}$  
    with~${I^{*} \supseteq I_0 \setminus \{ i \}}$ 
    and~${\mu(f {\upharpoonright} \beta) = \abs{I^{*}} - r(A {\upharpoonright} (I^{*} \times J^{*}))}$. 
    On the one hand, ${\abs{I^{*} \cup \{ i \}} = \abs{I^{*}} + 1}$ since~${i \notin I^{*}}$ and~${\mu(f) = \mu(f {\upharpoonright} \beta) + 1}$ since~${f(\beta) = i \in I}$. 
    On the other hand, ${r(A {\upharpoonright} (I^{*} \times J^{*})) = r(A {\upharpoonright} ((I^{*} \cup \{ i \})\times J^{*}))}$ 
    because Observation~\ref{obs: saturated} ensures that each column is extended only by a new~$0$ coordinate. 
    By combining these, we conclude that~${\mu(f) = \abs{I^{*} \cup \{ i \}} - r(A {\upharpoonright} ((I^{*} \cup \{ i \}) \times J^{*}))}$, thus~${I' := I^{*} \cup \{ i \}}$ and~${J' := J^{*}}$ are appropriate. 
    Now we suppose that~${j_0 := f(\beta) \in J}$. 
    The thin linear equation system
    \[ 
        \sum_{j \in \mathsf{ran}_J(f {\upharpoonright} \beta)} a_{i,j} x_j = a_{i, j_0} \quad (i \in I)
    \]
    has no solution since a solution would yield a non-trivial solution of~(\ref{eq: thinLES}). 
    Note that for~${i \in I \setminus \mathsf{ran}(f {\upharpoonright} \beta)}$, the equations above are trivial (i.e.~all coefficients and the right side are zeroes) by Observation~\ref{obs: saturated}. 
    Therefore the subsystem of the equations corresponding the indices in~${\mathsf{ran}_I(f {\upharpoonright} \beta)}$ is unsolvable. 
    By Theorem~\ref{thm: compact LES}, there is already a finite~${I_1 \subseteq \mathsf{ran}_I(f {\upharpoonright} \beta)}$ such that the corresponding subsystem is unsolvable. 
    Now we apply the induction hypothesis for~${f {\upharpoonright} \beta}$ and~${I_0 \cup I_1}$ to pick finite sets~${I^{*} \subseteq \mathsf{ran}_I(f {\upharpoonright} \beta)}$ and~${J^{*} \subseteq  \mathsf{ran}_J(f {\upharpoonright} \beta)}$  with~${I^{*} \supseteq 
    I_0 \cup I_1}$ and~${\mu(f {\upharpoonright} \beta) = \abs{I^{*}} - r(A {\upharpoonright} (I^{*} \times J^{*}))}$. 
    On one hand,~${\mu(f) = \mu(f {\upharpoonright} \beta) - 1}$ since~${f(\beta) = j_0 \in J}$. 
    On the other hand, ${r(A {\upharpoonright} (I^{*} \times (J^{*} \cup \{ j_0 \}))) = r(A {\upharpoonright} (I^{*} \times J^{*})) + 1}$ because the new column is not spanned by the old ones because~${I_1 \subseteq I^{*}}$. 
    By combining these, we conclude that~${\mu(f) = \abs{I^{*}} - r(A {\upharpoonright} (I^{*} \times (J^{*} \cup \{ j_0 \})))}$, thus~${I' := I^{*}}$ and~${J' := J^{*} \cup \{ j_0 \}}$ are appropriate, which completes the proof. 
\end{proof}

Let us restate our main theorem using the notation from Section~\ref{sec: prelims}. 

\setcounter{thm-intro}{0}
\begin{thm-intro}
    Let~$I$ and~$J$ be sets, let~$\mathbb{F}$ be a field, and let~$A \colon I \times J \to \mathbb{F}$ be a row-thin matrix. 
    If the homogeneous thin linear equation system~${Ax = 0}$ has only the trivial solution, 
    then there is an injection~${\varphi \colon J \to I}$ such that~${a_{\varphi(j),j} \neq 0}$ for every~${j \in J}$. 
\end{thm-intro}

\begin{proof}
    Suppose for a contradiction that the desired injection does not exist, i.e.~there is no matching in~$G_A$ that covers~$J$. 
    By Theorem~\ref{thm: jerzy}, we can pick a saturated string~$f$ with~${\mu(f) < 0}$. 
    By replacing~$f$ with an initial segment of itself if necessary, we can assume that~$\mu$ takes 
    only non-negative finite values on the proper initial segments of~$f$. 
    According to Lemma~\ref{lem: main}, we have~${\mu(f) = \abs{I'} - r(A {\upharpoonright} (I'\times J'))}$ for some finite~${I' \subseteq I}$ and~${J' \subseteq J}$. 
    Since the rank of a matrix is at most the number of rows this leads to~${\mu(f) \geq 0}$, contradicting the choice of~$f$. 
\end{proof}

Let us now prove Corollary~\ref{cor: diagonal}, which we restate using the terminology from Section~\ref{sec: prelims}. 

\begin{cor-intro}
    Let~$I$ and~$J$ be sets, let~$\mathbb{F}$ be a field, and let~${A \colon I \times J \to \mathbb{F}}$ such that the family of rows and the family of columns are both thinly independent. 
    Then there is a bijection~${\psi \colon J \to I}$ such that~${a_{\psi(j),j} \neq 0}$ for every~${j \in J}$. 
\end{cor-intro}

\begin{proof}
    Consider the graph~$G_A$ as above. 
    By Theorem~\ref{thm: main}, 
    since the family of rows is thinly independent there is matching that covers~$I$, and since the family of columns is thinly independent, there is a matching that covers~$J$. 
    By Theorem~\ref{thm: cantor-bernstein}, $G_A$ admits a perfect matching~$M$. 
    Setting~$\psi(j)$ to be the unique~${i \in I}$ for which~${\{i,j\} \in M}$ completes the proof. 
\end{proof}

\printbibliography

\end{document}